\documentclass[]{interact}

\usepackage{epstopdf}
\usepackage[caption=false]{subfig}
\usepackage[colorlinks = true,
linkcolor = blue,
anchorcolor = red,
citecolor = blue,
filecolor = red,
pagecolor = red,
urlcolor = blue]{hyperref}

\usepackage{amsmath,amsfonts,amsthm}
\usepackage{amssymb,latexsym}
\usepackage{graphicx}
\usepackage{mathtools}
\usepackage{bbm}
\usepackage{layout}
\usepackage{stmaryrd}

\usepackage[numbers,sort&compress]{natbib}
\bibpunct[, ]{[}{]}{,}{n}{,}{,}
\makeatletter
\def\NAT@def@citea{\def\@citea{\NAT@separator}}
\makeatother

\swapnumbers
\theoremstyle{plain}
\newtheorem{thm}{Theorem}[section]
\newtheorem{lem}[thm]{Lemma}
\newtheorem{cor}[thm]{Corollary}
\newtheorem{prop}[thm]{Proposition}
\newtheorem{defn}[thm]{Definition}

\theoremstyle{definition}
\newtheorem{claim}[thm]{}
\newtheorem{exam}[thm]{Example}

\theoremstyle{remark}
\newtheorem{rem}[thm]{Remark}
\newtheorem{nota}[thm]{Notation}

\newcommand{\bw}{\bigwedge}
\newcommand{\w}{\wedge}
\renewcommand{\uphi}{\underline{\phi}}
\newcommand{\End}{\mathrm{End}}
\newcommand{\ovD}{\overline{D}}
\newcommand{\OvD}{\overline{\D}}
\newcommand{\N}{\mathbb{N}}
\newcommand{\Z}{\mathbb{Z}}

\newcommand{\sra}{\rightarrow}
\newcommand{\bfe}{\mathbf{e}}

\newcommand{\lra}{\longrightarrow}
\newcommand{\tr}{\mathrm{tr}}

\newcommand{\KK}{\mathbb{K}}

\newcommand{\D}{\mathcal{D}}
\newcommand{\E}{\mathcal{E}}

\newcommand{\llb}{\llbracket}
\newcommand{\rrb}{\rrbracket}
\newcommand{\uz}{\underline{z}}

\renewcommand{\d}{\mathrm{d}}
\renewcommand{\a}{\underline{\Acal}}
\newcommand{\m}{{\underline{i}}}
\newcommand{\Acal}{\mathcal{A}}
\newcommand{\Bcal}{\mathcal{B}}

\begin{document}


\title{Multivariate Hasse--Schmidt Derivation on Exterior Algebras}

\author{
	\name{Fereshteh Bahadorykhalily} \thanks{\href{mailto:f.bahadori.khalili@gmail.com}{f.bahadori.khalili@gmail.com}}
\affil{Institute of Advanced Studies in Basic Science, Zanjan, Iran}
}

\maketitle

\begin{abstract}
The purpose of this short note is to consider multi-variate Hasse-Schmidt
derivations on exterior algebras and to show how they easily provide
remarkable identities, holding in the algebra of square matrices, which
generalise the classical theorem of Cayley-Hamilton.
\end{abstract}

\begin{keywords}
Hasse--Schmidt derivation on exterior algebra, traces of ordered tuples of endomorphisms, Generalised Cayley-Hamilton theorem, Exterior algebra.
\end{keywords}
\begin{amscode}
	15A75
\end{amscode}

\section{Introduction}\label{sec2}
\claim{} The main goal of this paper is to propose a new notion of {\em traces} of an $n$-tuple of endomorphisms of an $n$-dimensional vector space relying on 
the notion of multivariate Hasse--Schmidt (HS) derivation on an exterior algebra. We apply this notion to prove a Cayley-Hamilton (CH) like identity holding for $n$-tuples of endomorphisms
(Theorem \ref{thm:thm47}). It gives back the classical CH theorem for $n=1$. 

HS-derivations on an exterior algebra were defined  years ago by Gatto in \cite{SCHSD} and further 
developed in \cite{SCGA}, up to the most recent applications as 
in \cite{GatSch, pluckercone}. See also the book \cite{HSDGA} for a systematic presentation. 

If $V$  is an $n$-dimensional vector space, a HS--derivation on the exterior algebra $\bw V$ of $V$ is a 
map $\D(z):\bw V\sra \bw V\llb z\rrb$ such that $\D(z)(u\w v)=\D(z)u\w \D(z)v$. 
As remarked in the quite recent preprint \cite{behzad}, if $\D_1(z_1),\ldots, 
\D_k(z_k)$ ($k>1$) are HS-derivations,  their product is a  multivariate HS-derivation in the sense of this paper,  namely an $\End_\KK(\bw V)$-valued   
formal power series in the $n>1$ indeterminates $\uz:=(z_1,\ldots,z_n)$, with 
coefficients in $End_\KK(\bw V)$, but of a very special kind with respect to those considered here.

More precisely (see Definition \ref{def3:multv}), a   {\em multivariate} HS-derivation on
$\bw V$ is an $\End_\KK(\bw V)$-valued formal power series
$$
\D(\uz):=\sum_{\m\in\N^n}D_\m \uz^\m,
$$
such that $\D(\uz)(u\w v)=\D(\uz)u\w \D(\uz)v$. For $n=1$ and $z:=z_1$, this is 
evidently \cite[Definition 2.1]{SCHSD}. The
purpose of this note is then to extend some elementary properties, holding for 
univariate HS-derivations on exterior algebras, to the multivariate ones. 
Special emphasis will be put on what in \cite{SCHSD} and 
\cite{SCGA,pluckercone,HSDGA} is referred to as {\em integration by parts} (see 
Proposition~\ref{prop1}). In particular,  we will learn how to attach  scalars, 
that we think appropriate to call {\em traces}, to a finite sequence of 
endomorphisms. This in turn generalises the main result of \cite{GatSch}.

\claim{} A quick description of our main result follows. Let $\uphi:=(\phi_1,
\ldots,\phi_n) \in \End_\KK(V)^n$ be an ordered $n$-tuple of endomorphisms of $V$. By 
Proposition \ref{prop3:36}, there is one and only one multivariate HS--derivation $$\OvD(\uz)=\sum_{\m\in\N^n}(-1)^{|\m|}\ovD_\m z^\m,$$
such that
$
\OvD(\uz)_{|_V}=1-(\phi_1z_1+\cdots+\phi_nz_n).
$
The  $\m$--{\em trace} $\tau_{\m}(\uphi)$ is the scalar defined by
$$
\tau_{\m}(\uphi)\xi=\ovD_{\m}\xi,
$$
where  $\xi$ is any non zero element of $\bw^nV$. If $n=1$, then $\tau_{(i)}
(\phi)$ are precisely the {\em traces} as defined in \cite{GatSch},  which in 
turn coincide with the coefficients of the characteristic polynomial $\det(t
\cdot\mathrm{id}_V-\phi)$ of $\phi$ (e.g. $\tau_{(n)}$ is the determinant).
For example, if $A=(a_{ij})$ and $B=(b_{ij})$ are $2\times 2$  matrices thought 
of as endomorphisms of $\KK^2$,  then $\tau_{(0,0)}(A,B)=1$, $\tau_{(1,0)}
(A,B)=\tr(A)=a_{11}+a_{22}$, $\tau_{(0,1)}(A,B)=\tr(B)=b_{11}+b_{22}$, $
\tau_{(2,0)}(A,B)=\det(A)$,  $\tau_{(0,2)}(A,B)=\det(B)$
and, finally
$$
\tau_{(1,1)}(A,B)=\begin{vmatrix}a_{11}&b_{12}\cr a_{21}&b_{22}\end{vmatrix}+\begin{vmatrix}b_{11}&a_{12}\cr b_{21}&a_{22}\end{vmatrix}.
$$
Clearly $\tau_{\m}(\uphi)=0$ for all $\m$ such that $|\m|\geq i+1$. See Example \ref{extraces} for the explicit expression of all the traces $\tau:\N^3\sra \KK$ of triples of $3\times 3$ matrices. 

\medskip
\noindent
The notion of $\m$--trace enables us to state the following:

\smallskip
\noindent
{\bf MAIN THEOREM  \ref{thm:thm47}.} {\em For all $n$-tuples of endomorphisms of $V$,  the identity:
	$$
	\sum_{k=0}^n(-1)^k\dfrac{1}{k!}\sum_{\sigma\in S_n}\tau_{\bfe_{\sigma(1)}+\cdots+\bfe_{\sigma(k)}}(\uphi)\cdot (\phi_{\sigma(k+1)}\circ \ldots  \circ \phi_{\sigma(n)})=0,
	$$
	holds for all $\uphi\in \End_\KK(V)^n$, 
	where $(\bfe_1,\ldots,\bfe_n)$ denotes the canonical basis of $\Z^n$ 
	and $S_n$  the symmetric group on $n$ letters.
}

\smallskip
By their very definition,  the scalars $\tau_{\m}(\uphi)$ are invariants of $\uphi$, i.e. they are attached to the endomorphisms themselves and not to the particular bases one uses to represent them. This fact  can be easily checked by hand, namely: {\em the traces of $n$-tuples of matrices only depend on their conjugacy classes}. The practical output is that the multi--dimensional array $(\tau_{\m}(\uphi))$   can be computed in terms of matrix products upon replacing each $\phi_i$ by its matrix with respect to a fixed $\KK$-basis. In particular we have 

\medskip
\noindent
{\bf Corollary.} {\em  For all ordered $n$-tuples of  $\KK$-valued $n\times n$ matrices $\a=(A_1,\ldots, A_n)$, the following identity holds:
	$$
	\sum_{k=0}^n(-1)^k\dfrac{1}{k!}\sum_{\sigma\in S_n}\tau_{\bfe_{\sigma(1)}+\cdots+\bfe_{\sigma(k)}}(\a)\cdot (A_{\sigma(k+1)} \cdots A_{\sigma(n)})=0,
	$$
where $\tau_{\m}(\a)$ denotes the trace of the $n$-tuples of
matrices thought of as endomorphisms of $\KK^n$.
}

\medskip
\noindent

A few comments are in order. First of all, if $A_1=\cdots=A_n=A$, and  provided we are in characteristic zero, the theorem gives back the Cayley--Hamilton identity,  according to which each matrix is a root of its own characteristic polynomial. Secondly, as mentioned at the very beginning of this introduction, the theorem can be related to some trace identity relations as in \cite{procesi} and it would be interesting to see how much of the formalism of the present paper can be generalized to deal with the more general situations studied there.
As a matter of example, if $(A,B)\in (\KK^{2\times 2})^2$ are as above, comparing the coefficients of $z_1^2$ in the integration by parts formula, one obtains the identity
\[
\tr(A^2)+2\det(A)-\tr(A)^2=0,
\]
which is evidently true if $A$ is diagonal.

\smallskip
\claim{} The Main Theorem also suggests a number of amusing corollaries.
Denote by $C_i(A)$ the $i$-th column of a matrix $A$. Then our main result shows that the bilinear map $\star:\KK^{2\times 2}\times \KK^{2\times 2}\sra \KK$ given by
$$
A\star B=AB-a_{11}B-b_{22}A+\det(C_1(A),C_2(B))\mathbbm{1}_{2\times 2}
$$
is skew symmetric, i.e.,
\begin{equation}\label{eq:skws}
A\star B+B\star A=0.
\end{equation}
An easy check shows however that $(\KK^{2\times 2}, \star)$ is not a Lie algebra.
Putting $A=B$ one obtains
$$
0=A\star A=2\cdot \left(A^2-\tr(A)A+\det(A)\mathbbm{1}_{2\times 2}\right)=0
$$
that explains in which sense formula \eqref{eq:skws} generalises Cayley-Hamilton theorem for $2\times 2$ matrices, as announced. 

\noindent
Similarly, let $\a=(A,B,C)$, where  $A=(a_{ij})$, $B=(b_{ij})$, $C=(c_{ij})$ are $3\times 3$ matrices.
Then the tri-linear map 
$
\KK^{3\times 3}\times \KK^{3\times 3}\times \KK^{3\times 3}\sra \KK^{3\times 3}
$ 
defined by
\begin{align*}
A\star B\star C&=ABC\\
&-a_{11}BC-b_{22}CA-c_{33}AB\\
&+(a_{11}b_{22}-a_{21}b_{12})C+(b_{22}c_{33}-b_{32}c_{23})A+(a_{11}c_{33}-a_{31}c_{13})B\\&-\det(A\bfe_1,B\bfe_2, C\bfe_3)\mathbbm{1}_{3\times 3}
\end{align*}
is skew-symmetric (Cf. Example \ref{ex:exam52}), i.e.,
$$
s(A)\star s(B)\star s(C)=(-1)^{|s|}A\star B\star C
$$ 
for all bijections $s:\{A,B,C\}\sra \{A,B,C\}$, where by $|s|$ we have denoted the sign of the permutation. Again $A\star A\star A=0$ phrases the content of the  Cayley--Hamilton theorem for $3\times 3$ matrices.

\claim{} Some work  concerning the study of the properties of 
the function $\tau_{\uphi}:\N^n\sra \KK$, associating to each multi-index  $\m
\in \N^n$ the multidimensional matrix of their traces, is in progress.  Especially we are guessing applications to the theory of the generic linear PDEs with constant coefficients, in the same fashion as in the paper \cite{GatSch1}, and to the generalisation of the formalism proposed in \cite{HSDGA, pluckercone},
to produce more general kinds of vertex operators, naturally associated to 
multivariate HS--derivations of a more general nature than the very special ones 
already dealt with in \cite{behzad}.

\noindent

\claim{} The paper is organized as follows. In Section \ref{sec:preli}, we 
recall some preliminaries on HS--derivations on exterior algebras, as in e.g. 
\cite{SCHSD, HSDGA}, to keep this note as self contained as possible.  We 
also recall, in this section,  the statement of \cite[Theorem 2.3]{GatSch}  
concerning a Cayley--Hamilton vanishing theorem holding on the exterior 
algebra of a vector space. Section \ref{sec:sec3} sets a minimum of 
foundational material regarding multi-variate HS--derivations, putting again 
a special emphasis on the integration by parts formula. Section 
\ref{sec4:main}, instead, gathers together the main definitions, lemmas and 
propositions recalled up to then, enabling us to formulate and prove the  
Main Theorem \ref{thm:thm47}. A few applications and examples, with the 
purpose of making our main result concrete and explicit, are listed in 
Section \ref{sec:sec5}. These examples may also suggest to the reader how to 
prove the Theorem~\ref{thm:thm47} in special cases and also served us as main 
sources of experiments.

\section{Preliminaries and Notation}\label{sec:preli} 
The main general references for this section, where the interested reader can find the omitted details, are \cite[Section 2]{pluckercone} and/or \cite{HSDGA}.

Let
$V$
be a vector space of dimension $n$ over a field $\KK$, and let 
$(u_1, \cdots , u_{n})$
be one  $\KK$-basis. 
We denote by $\bigwedge V $ the exterior algebra of $V$. It is a direct sum $\bigoplus_{j\geq 0}\bigwedge^j V$ of exterior powers of $V$. Recall that
$\bigwedge^0 V=\KK$ while, 
if $j>0$, the vector space $\bigwedge^j V$
has  a basis
$u_{i_1} \wedge \cdots \wedge u_{i_j}$ with 
$i_1 < i_2 < \cdots <i_k$.
If 
$\gamma \in S_j$,
the symmetric group of 
$\{1,2,\cdots,j\}$,
then
$u_{i_1} \wedge \cdots \wedge u_{i_j} =sgn(\gamma)u_{i_{\gamma(1)}} \wedge 
\cdots \wedge u_{i_{\gamma(j)}}$.

Let
$\bigwedge V \llb z \rrb$
be the $ \bigwedge V$-valued formal power series with its natural algebra 
structure.
\begin{defn} (See \cite[Definiton 2.1]{SCHSD}). 
A HS--derivation on the exterior algebra $ \bigwedge V $ is a 
$\KK$-linear map 
$\D(z): \bigwedge V \longrightarrow \bigwedge V \llb z \rrb ,$
such that for $ u,v\in \bigwedge V $
\begin{equation}\label{first}
\D(z)(u\wedge v)= \D(z) u \wedge \D(z) v.
\end{equation}
\end{defn}
We shall write $\D(z)$ in the form $\sum\limits_{i\geq 0}D_{i}z^i$, where 
$D_{i}\in \End_\KK(\bw V)\llb z\rrb$.
\begin{prop} Let $\overline{\D}(z)\in End_\KK(\bw V)$ such that $\overline{\D}
(z)\cdot \D(z)=1$. Then it is a HS--derivation, said to be the {\em inverse} of 
$\D(z)$.
\end{prop}
If $ D_0 $ is invertible in $\End_\KK(\bw V)\llb z\rrb$, then $\D(z)$ is invertible 
in $\bigwedge V\llb z \rrb$. We shall write 
$\D(z) = \sum\limits_{i \geq 0} D_i z^i$ and $ \OvD(z)= \sum\limits_{i \geq 0} 
(-1)^i\ovD_i z^i$, 
with $ D_i,\ovD_i\in End_\KK(\bigwedge V) $.

\claim{} Let  $HS(\bigwedge V)$ be the set of all HS--derivations on
$\bigwedge V$.
It is easily seen that the invertible HS--derivations on
$\bigwedge V$ form a group under
composition.
Formula \eqref{first} is equivalent to the fact that each $D_i$ satisfies the Leibniz rule, i.e for all
$u,v \in \bigwedge V$ \cite[Proposition 4.1.5]{SCHSD}
\[ D_i (u \wedge v)= \sum\limits_{k=0}^i D_ku \wedge D_{i-k}v.\]
\begin{prop}\label{prop:prop23}
Let $f\in End_{\KK}(V)$ and let 
$f(z)=\sum\limits_{i \geq 0} (-1)^if^i z^i: V \rightarrow  V\llb z \rrb.$
Then there is a	unique HS-derivation
$ \D^f(z)\in HS(\bigwedge V) $
such that $ \D^f(z)|_{V}=f(z) $.
\end{prop}
\begin{proof}
The proof of \cite[Page 68]{HSDGA} will be generalised to the multivariate 
case in the Proposition \ref{prop3:36} below.
\end{proof}


\claim{\rm } Within the same situation of Proposition \ref{prop:prop23}, denote 
by $\OvD^f(z)$ the  inverse of $\D^f(z)$. For each $i\geq 0$, let $e_i$ be the 
eigenvalue of $\ovD^f_i$ thought of as an element of $End_\KK(\bigwedge^nV)$. 
\begin{thm}[\cite{GatSch}]
If $f\in End_\KK(V)$ and $\D(z)$ is the corresponding  HS--derivation on 
$\bw V$, then the equality
\begin{equation}
D_n^f-e_1D_{n-1}^f+\cdots+(-1)^ne_n{\mathbbm 1}=0, \label{eq:gsch}
\end{equation}
holds in $\End_\KK(\bw V)$.
\end{thm}

Equality \eqref{eq:gsch} above implies the classical theorem by Cayley and 
Hamilton, by restricting to $V$, i.e.
\[
0=(D_n^f-e_1D_{n-1}^f+\cdots+(-1)^ne_n{\mathbbm 1})_{|V}=f^n-e_1f^{n-1}+\cdots+(-1)e_nf=0.
\]
Indeed, one can easily check that the eigenvalues of $\ovD^f_{|\bw^nV}$ are precisely the coefficients of the characteristic polynomial of $f$.
\section{Multivariate HS--derivations on $\bw V$}\label{sec:sec3}
\claim{}  Let 
$ \bw V\llb\uz\rrb:=\bigwedge V\llb z_1 , \cdots , z_n \rrb $
be the algebra of formal power series in $n$ formal  variables
$\uz:=( z_1 , \cdots , z_n) $ with values in the exterior algebra of the $n$-dimensional $\KK$-vector space $V$. If $\m:=(i_1,\ldots,i_n)\in\Z^n$ is an $n$-tuple of non negative integers, we write $\uz^{\m}=z_1^{i_1}\cdots z_n^{i_n}$ and $|\m|=i_1+\cdots+i_n$, in such a way that $\OvD(\uz)=\sum_{\m\in \N^{\m}}(-1)^{|i|}\ovD_{\m}\cdot \uz^{\m}$.

\begin{defn} \label{def3:multv}
 An $n$--multivariate HS--derivation is an algebra homomorphism
\begin{equation*}
\D(\uz):\bigwedge V\longrightarrow \bigwedge V\llb z_1,\cdots, z_n \rrb
\end{equation*}
in the sense that for every
$u,v \in \bigwedge V$,
$ \D (\underline{z})(u \wedge v) = \D(\underline{z})u \wedge \D(\underline{z}) v $.
\end{defn}
Denote by $HS_n(\bigwedge V)$ the set of all $n$--multivariate HS--derivations on $ \bigwedge V $. Each element   $ \D(\underline{z})\in HS_n(\bigwedge V) $ can be written as 
\[
\D(\underline{z})= \sum_{\m} 
D_{\m} \uz^\m, \qquad D_{\m} \in \End_\KK(\bw V).
\] 
\claim{}
The  product 
\[
\D(\uz)\E(\uz)=\sum \D(\uz) E_{\m} \uz^\m,
\]
where
$$\D(\uz) E_{\m}=\sum\limits_{\underline{j}}
\big( D_{\underline{j}}
\circ E_{\m}\big) \uz^{\underline{j}},$$
makes $HS_n(\bigwedge V) $ into a multiplicative monoid. The identity is the identity map of $\bw V$ thought of as a constant HS-derivation.
Moreover,  if $\D(\uz)$ and $\E(\uz)$ are in $HS_n(\bigwedge V)$, then $\D(\uz)\E(\uz)$ is  an element of $HS_n(\bigwedge V)$, because  for all $ u,v \in \bigwedge V $:
\begin{equation}\label{for1}
\begin{split}
&\D(\uz)\E(\uz) (u\wedge v) 
= \D(\uz)
\sum\limits_{\underline{l}} \sum\limits_{\underline{i}+\underline{j}=\underline{l}}\bigg( E_{\m}u \wedge E_{\underline{j}}v \Big)\uz^{\underline{l}} \\[3pt]
&= \sum\limits_{\underline{l}} \sum\limits_{\underline{i}+\underline{j}=\underline{l}} \D(\uz)E_{\m}u\cdot\uz^\m \wedge \D(\uz) E_{\underline{j}}v\cdot\uz^{\underline{j}}  \\[5pt]
&= \D(\uz)\E(\uz)u\wedge \D(\uz)\E(\uz)v.
\end{split}
\end{equation}
The next lemma basically says that the coefficients of a multivariate  $HS$--derivation on $\bw V$ behave as partial derivatives in ordinary 
calculus.
\begin{lem}
Let
$ \D(\uz) = \sum_{\m} D_{\m}\uz^\m\in  End_\KK(\bigwedge V)\llb \uz \rrb $. 
Then the equalities below are equivalent.
\begin{enumerate}
\item \label{it1} $ \D(\uz) (u \wedge v) = \D(\uz) u \wedge \D(\uz) v$.
\smallskip
\item \label{it2} $ D_{\m} (u \wedge v)= 
\Big(\sum_{\underline{j}+\underline{l}=\m} D_{\underline{j}} u \wedge 
D_{\underline{l}} v \Big)\uz^\m$.
\end{enumerate}
\end{lem}
\begin{proof}
	\eqref{it1} $ \Rightarrow $ \eqref{it2} is clear since \eqref{it1} means
	\[\D(\uz) (u \wedge v)= \sum_{\underline{j}} D_{\underline{j}}u\cdot \uz^{\underline{j}}
	\wedge
	\sum\limits_{\underline{l}} D_{\underline{l}}v\cdot\uz^{\underline{l}}.
	\]
	Thus, the coefficient of $\uz^\m$ in the right hand side  of (\ref{it1}) is precisely $\displaystyle{\sum_{\underline{j}+\underline{l}=\m}} D_{\underline{j}}(u) \wedge D_{\underline{l}} (v)$. Conversely, if  \eqref{it2} holds one has
	\begin{align*}
	\D(\uz) (u \wedge v)& =
	\sum\limits_{\m} D_{\m}(u\wedge v)\uz^\m
	= \sum\limits_{\m}
	\bigg(\sum\limits_{\underline{j}+\underline{l}=\m} D_{\underline{j}}u\wedge D_{\underline{l}}v\bigg)\uz^\m\\
	& =\sum\limits_{\m} \bigg(\sum\limits_{\underline{j}+\underline{l}=\m} D_{\underline{j}}u \cdot \uz^{\underline{j}}
	\wedge D_{\underline{l}} v \cdot \uz^{\underline{l}}\bigg)
	=\D(\uz)u \wedge \D(\uz)v. 
	\end{align*}
\end{proof}
\begin{exam}
	Let $(\bfe_1,\ldots, \bfe_n)$  be the canonical basis of ${\mathbb Z}^n$. 
It is easy to see that for every $0 \leq k \leq n$, the component
	$ D_{\bfe_k}:=D_{(0,\ldots,1,\ldots,0)}$ ($1$ as $k$-th  entry)
	is the coefficient of $z_k$. Each $D_{\bfe_k}$ is a usual derivation of the exterior algebra.
\end{exam} 

The following generalises \cite{GatSch} in case $n=1$. 

\begin{prop}\label{prop3:36} Let 
	$
	f(\underline{z})=\sum_{\m\in\N^n}f_{\m}\uz^\m: V \rightarrow  V\llb \underline{z} \rrb 
	$
	be an $\End_\KK(V)$-valued formal power series. Then there is a unique multivariate HS-derivation $\D^f(\uz)$ on $\bw V$ such that  
	$\D^f(\uz)|_{V}=f(\uz)$.
\end{prop}
\begin{proof}
For all $k\geq 0$, define the multilinear map
\begin{equation}
\left\{\begin{matrix}
\underbrace{V\otimes \cdots\otimes V}_{k\,\,times}&\lra& \bw^kV\llb\uz\rrb\cr\cr
v_1\otimes\cdots\otimes v_k&\longmapsto &f(\uz)v_1\w \cdots\w f(\uz)v_k
\end{matrix}\right.\label{eq36:multi}
\end{equation}
Thus, besides being clearly $\KK$-multi-linear, the map \eqref{eq36:multi} is also alternating. The universal property of the exterior powers implies that it factorises uniquely through a map $\D^{k,f}(\uz):\bw^kV\sra \bw^kV\llb\uz\rrb$. This provides the map
$\D^f(\uz)\in HS_n(\bw V)$. Indeed, any $u\in \bw V$ is a finite linear combination of homogeneous decomposable elements of the form $u_1\w\cdots\w u_k$. It is then enough to define $\D^f(\uz)$ for all the elements of this form, what  we do as follows:
$$
\D(\uz)(v_1\w\cdots\w v_k)=f(\uz)v_1\w\cdots\w f(\uz)v_k.
$$
This proves the existence, while the uniqueness is a straightforward exercise.
\end{proof}

\claim{} If $ \D_{(0, \cdots 0)}^f$ is invertible, then $\D^f(\uz)$ is invertible in $\End_\KK(\bw V)\llb \uz\rrb$. We write the inverse in the form
$$ \OvD^f(\uz)=  \sum_{k\geq 0}\sum_{|\m|=k} (-1)^{k}\ovD_{\m}^f\uz^\m.
$$
Clearly
$\ovD_{(0, \ldots , 0)}^f= D_{(0, \ldots , 0)}^f $ 
and an easy check provides the equality
$\ovD^f_{\bfe_k}= D_{\bfe_k}^f$ holding for all $k\geq 0$. 
The following is just a rephrasing of \cite{HSDGA}.
\begin{prop}
	The inverse $\OvD(\uz)$ of the HS--derivation $\D(z)$ is a HS--derivation.
\end{prop}
\begin{proof}
	Using 
	equation \eqref{for1},
	for every $u,v \in \bigwedge V$, we have
	$$
	\OvD(\uz)(u\wedge v) = \OvD(\uz)(\D(\uz)\OvD(\uz)u \wedge \D(\uz)\OvD(\uz) v)= \OvD(\uz) u \wedge \OvD(\uz)v.
	$$
	Therefore,
	$\OvD(\uz) \in HS_n(\bigwedge V)$.
\end{proof}
\begin{prop}[Integration by parts]\label{prop1}
	The following formulas hold for every $u,v \in \bigwedge V$:
	\begin{eqnarray}
	\D(\uz) u \wedge v &=&\D(\uz) (u \wedge \OvD(\uz) v), \label{for2.2}\cr \cr
	 \OvD(\uz)u\wedge v&=& \OvD(\uz)(u \wedge \D(\uz)v). \label{for2.3}
	\end{eqnarray}
\end{prop}
\begin{proof}
	It is straightforward  and works exactly the same as in \cite[Proposition 4.1.9]{HSDGA}.
\end{proof}

\claim{} From the fact that each $ v \in \bigwedge^k V $ is a finite linear combination of elements of the form
$ v_{i_1} \wedge v_{i_2} \wedge \cdots \wedge v_{i_k}$, where $v_{i_j}\in V$,  and 
$$
\D(\underline{z})(v_{i_1} \wedge v_{i_2} \wedge \cdots \wedge v_{i_k})=\D(\underline{z})v_{i_1}\wedge\cdots\wedge \D(\underline{z})v_{i_k},
$$
it follows that $\D(\uz)$ is uniquely determined by its restriction 
$ \D(\uz)|_{V}\colon V \longrightarrow \bw V\llb \uz \rrb $
to the first degree of the exterior algebra. In this note we shall stick to the case in which $\D(\uz)$ is homogeneous of degree $0$ with respect to the exterior algebra graduation, i.e., each coefficient of $\D(\uz)$ maps $\bw^kV$ to itself. In this case $ \D(\uz)\vert_{V} \in End_\KK(V)\llb \uz \rrb $. In the sequel we shall construct a relevant homogeneous multivariate HS-derivation of degree $0$.
\section{Traces of $n$-tuples of endomorphisms and the main result}\label{sec4:main}
\claim{}\label{sec41} 
Let  $\uphi\coloneqq(\phi_1,\ldots, \phi_n)$  be an ordered $k$-tuple of endomorphisms  of  our fixed $n$-dimensional $\KK$-vector space  $V$. Because of Proposition \ref{prop3:36}, there exists a  unique multivariate HS--derivation 
$\OvD(\uz):=\sum_{\m\in \N^{\m}}(-1)^{|i|}\ovD_{\m}\cdot \uz^{\m}:\bw V\sra \bw V \llb\uz\rrb$ such that
\begin{equation}\label{eq:inverse}
\OvD(\uz)_{|V}=\mathbbm{1}-(\phi_1 z_1+\cdots +\phi_n z_n).
\end{equation}

\begin{prop}\label{prop3} For all $u\in\bw^kV$, $\OvD(\uz)u\in \bw V[ \uz]$.  More precisely, it is  a polynomial in the indeterminates $(z_1,\cdots,z_n)$ of degree at most $k$.
\end{prop}
\begin{proof}
We use induction on the degree $k$ of the exterior algebra.
For $k=1$, we have
\[\OvD(\uz)u=u-\phi_1(u)z_1-\cdots-\phi_n(u)z_n,\]
which is a polynomial of degree one in the indeterminates $z_1,\ldots,z_n$. Suppose the property  holds true for all $1\leq l<k$ and $u \in \bigwedge^k V$. 
Without loss of generality, one may assume $u$  homogeneous of the form $u_1 \wedge u_2$, where $u_1\in \bigwedge^1V=V$ and $u_2\in\bigwedge^{k-1}V$. Thus
\[\OvD(\uz)u=\OvD(\uz)(u_1 \wedge u_2)=
\OvD(\uz)u_1 \wedge \OvD(\uz)u_2.
\]
Since  $\OvD(\uz)u_1$ is a polynomial of degree one and,  by the inductive hypothesis, $\OvD(\uz)u_2$ is a polynomial of degree at most $k-1$, it follows that $\OvD(\uz)u$ is a polynomial of degree at most $k$.
\end{proof}
\begin{cor}\label{cor1}
For all $\m\in \mathbb{N}^n$ such that $|\m| >j$, 
${\ovD_{\m}}_{|\bigwedge^jV}=0$.
In other words 
\begin{equation*}
|\m|>j\quad\Rightarrow\quad \bigwedge ^jV \subset \text{ker}(\ovD_{\m}).
\end{equation*}
\end{cor}
\begin{proof}
	Let $u\in\bigwedge^j V$. Then we know $\ovD_{\m}u$ is the coefficient of 
	$\uz^\m$
	in $\OvD(\uz)u$. Now if $|\m| >j$, by Proposition \eqref{prop3} the coefficient of $\uz^\m$ in $\OvD(\uz)u$ is zero.
\end{proof}
\begin{defn}\label{def:trace1}
The $\m$-trace of an $n$-tuple of endomorphisms $\uphi$ on $V$ is  the scalar  $\tau_{\m}(\uphi)$ defined through the equality
\begin{equation}\label{key}
\tau_{\m}(\uphi)\cdot \xi=\ovD_{\m}\xi,
\end{equation} 
for each $\xi\neq 0\in \bw^nV$.
\end{defn}
In other words, $\bw^nV$ is an eigenspace of each $\ovD_{\m}$ with respect to the eigenvalue $\tau_{\m}(\uphi)\in\KK$.
\begin{exam}\label{ex:ex4.5}
Let $V$ be a vector space of dimension $2$ and $\uphi\coloneqq(\phi_1,\phi_2)$ be a pair of endomorphisms of $V$.
Let $\OvD(\uz)$ be the unique multivariate HS--derivation on $\bw V $ such that 
$\OvD(z_1,z_2)_{|V}=\mathbbm{1}-\phi_1z_1-\phi_2z_2$. It is invertible. Let
${\mathcal D}(\uz)=\sum_{\m\in\N^2}D_\m\uz^\m$
 be its inverse. Then we have
\begin{equation*}
\begin{split}
\OvD(z_1,z_2)(u,v)&=\OvD(z_1,z_2)u\w\OvD(z_1,z_2)v\\[5pt]
& = (u-\phi_1uz_1-\phi_2uz_2)\w(v-\phi_1vz_1-\phi_2vz_2)\\[5pt]
& = u\w v-(\phi_1u\w v+u\w \phi_1v)z_1 -(\phi_2u\w v+u\w \phi_2v)z_2\\[5pt]
& +( \phi_1u\w \phi_1v) z_1^2+(\phi_2u\w \phi_2v) z_2^2 + (\phi_1u\w \phi_2 v+\phi_2 u\w \phi_1 v)z_1z_1 +\cdots
\end{split}
\end{equation*}
	The integration by parts formula \eqref{for2.2} states that
	\[
	\OvD(z_1,z_2)(\D(z_1,z_2)u\w v)= u\w \OvD(z_1,z_2)v.
	\]
	The coefficients of $z_1z_2$ in the right hand side is $u\w \ovD_{(1,1)}v$, which is zero because of Proposition~\ref{prop3}. Thus the coefficient of $z_1z_2$ of the left hand side is zero, too. Explicitly:
	\begin{equation*}
	D_{(1,1)}u\w v-\ovD_{(0,1)}(D_{(1,0)}u\w v)-\ovD_{(1,0)}(\D_{(0,1)}u\w v)+\ovD_{(1,1)}(u\w v)=0,
	\end{equation*}
	which can be written
	\begin{equation*}
	D_{(1,1)}u\w v-\tau_{(0,1)}(\uphi)(D_{(1,0)}u\w v)-\tau_{(1,0)}(\uphi)(\D_{(0,1)}u\w v)+\tau_{(1,1)}(\uphi)(u\w v)=0.
	\end{equation*}
	Since $\tau_{\m}(\uphi)\in \KK$ and because of the multilinearity of the wedge product, we get
	\begin{equation*}
	\left(D_{(1,1)}u - \tau_{(0,1)}(\uphi)D_{(1,0)}u-\tau_{(1,0)}(\uphi)D_{(0,1)}u+\tau_{(1,1)}(\uphi)u\right)\w v=0.
	\end{equation*}
	i.e.
		\begin{equation*}
	D_{(1,1)}u - \tau_{(0,1)}(\uphi)D_{(1,0)}u-\tau_{(1,0)}(\uphi)D_{(0,1)}u+\tau_{(1,1)}(\uphi)u=0,
	\end{equation*}
due to the arbitrary choice of $v\in V$.

From the fact that $\D(\uz)$ and $\OvD(\uz)$ are mutually inverses in $\bw V\llb z_1,z_2\rrb$,  a simple computation, keeping into account the initial conditions, yields:
	$$D_{(1,0)}=\phi_1 , D_{(0,1)}=\phi_2 , D_{(1,1)}=\phi_1\circ\phi_2 +\phi_2\circ\phi_1$$
	Therefore, we have 
	\begin{equation}\label{eq:ch2}
	\phi_1\circ\phi_2 +\phi_2\circ\phi_1 +
	\tau_{(0,1)}(\uphi)\phi_1-\tau_{(1,0)}(\uphi)\phi_2+\tau_{(1,1)}(\uphi)\mathbbm{1}_{2\times 2}=0.
	\end{equation} 
Let us see the explicit expression $\tau_{(i,j)}(\uphi)$ in terms of  the pair $\uphi$. We know that
	\begin{equation*}
	\ovD_{(1,0)}(u\w v)=\ovD_{(1,0)}u\w v+u\w \ovD_{(1,0)}v =\phi_1 u\w v+u\w \phi_1v.
	\end{equation*}
i.e.
$
\tau_{(1,0)}(\uphi)(u\w v)=\tr(\phi_1)(u\w v)
$, according to the usual notion of trace of an endomorphism. Similarly
\[
\tau_{(0,1)}(\uphi)(u\w v)=\tr(\phi_2)(u\w v).
\]
It is also easily seen that
\[
\ovD_{(1,1)}(u\w v)=\ovD_{(1,0)}u\w\ovD_{(0,1)}v+\ovD_{(0,1)}u\w\ovD_{(1,0)}v=\phi_1 u\w \phi_2 v+ \phi_2 u\w \phi_1 v,
\]
from which
\[\tau_{(1,1)}(\uphi)(u\w v)=\phi_1(u)\w \phi_2(v)+ \phi_2(u)\w \phi_1(v).\] 
\end{exam}
\begin{exam} Let $\a=(A_1,A_2)$  be  matrices of the pair $\uphi:=(\phi_1,\phi_2)$ of Example \ref{ex:ex4.5}, with respect to some arbitrary basis $ (u,v) $ of $V$.
If  $A_1$ and $A_2$ are given by:
\[ A_1=\begin{bmatrix}
a & b\\ c &d
\end{bmatrix} \qquad \mathrm{and}\qquad  A_2=\begin{bmatrix}
\alpha & \beta \\ \gamma & \delta
\end{bmatrix}, \]
then
\begin{equation}\label{det1}
\begin{split}
\tau_{(1,1)}(\uphi)(u\w v)&=\left[\det(\phi_1 u,\phi_2 v)+\det(\phi_1 u,\phi_2 v)\right](u\w v)\\[5pt]
&= \left(\det \begin{bmatrix}
a & \beta \\ b & \delta
\end{bmatrix}+
\det \begin{bmatrix}
\alpha & b \\ \gamma & d
\end{bmatrix}\right)(u\w v).
\end{split}
\end{equation}
Also we have
\[\tau_{(1,0)}(\uphi)(u\w v)=\tr(A_1)(u\w v),\quad \tau_{(0,1)}(\uphi)(u\w v)=\tr(A_2)(u\w v).\]
In case $\phi_1=\phi_2$, and $A$ is its matrix with respect to the basis $(u,v)$ of $V$, formula \eqref{det1} shows that 
\[\tau_{(1,1)}(\uphi)=2\det(A).\]
So in this case, the formula \eqref{eq:ch2} becomes
\[2\left(A^2-\tr(A)A+\det(A)\mathbbm{1}_{2\times 2}\right)=0,\]
which is the classical Cayley-Hamilton theorem.
\end{exam}
\begin{rem}
Let $\uphi=(\phi_1,\cdots,\phi_k)\in \End_\KK(V)^n$ as at the beginning of Section~\ref{sec41}. Let $(u_1,\cdots, u_n)$ be a basis of $V$. 
It easily follows from Definition \ref{def:trace1} that 
\begin{equation}
\begin{split}
&\tau_{\bfe_{i_1}+\cdots+\bfe_{i_k}}(\uphi)(u_1\w\cdots \w u_n)=\\
&
\sum\limits_{\sigma\in S_n}(-1)^{|\sigma|}
\det\left(\phi_{i_1}u_{\sigma(1)},\cdots,\phi_{i_k}u_{\sigma(k)},\cdots,u_{\sigma(n)}\right)(u_1\w \cdots\w u_n).
\end{split}
\end{equation}
In particular, if $ A_i $ is the  matrix of $ \phi_i$ with respect to the chosen basis, and if $A=A_1=\cdots=A_n$, then $\tau_{(1,1,\cdots,1)}(\uphi)=n!\det(A)$.
\end{rem}

\medskip
\noindent We pass now to state and prove the main result of this note.
\begin{thm}\label{thm:thm47}
Let $V$ and $\uphi=(\phi_1,\cdots,\phi_k)$ as in \ref{sec41}. Then the following identity holds
\begin{equation}\label{eq:formnthm}
\sum_{k=0}^n(-1)^k\dfrac{1}{k!}\sum_{\sigma\in S_n}\tau_{\bfe_{\sigma(1)}+\cdots+\bfe_{\sigma(k)}}(\uphi)\cdot (\phi_{\sigma(k+1)}\circ \ldots  \circ \phi_{\sigma(n)})=0,
\end{equation}
where by convention we set to  $\sum_{\sigma\in S_n}\phi_{\sigma(1)}\circ\cdots \circ\phi_{\sigma(n)}$  and 
$\tau_{\bfe_{\sigma(1)}+\cdots+\bfe_{\sigma(n)}}(\uphi)$
the summands corresponding to $k=0$ and $k=n$ respectively.
\end{thm}
\begin{proof}
Let 
$\OvD(\uz)\colon\bigwedge V \rightarrow \bigwedge V\llb \uz \rrb $
be the unique multivariate  HS--derivation such that 
\[\OvD(\uz)_{|V} =  \mathbbm{1} - \phi_1z_1 - \phi_2z_2 - \cdots - \phi_nz_n.\]
Let $\D(\uz)$ be the inverse of $\OvD(\uz)$ in $\bw V\llb \uz \rrb$, whose restriction  to $V$ clearly is
\begin{equation}\label{eq4:14}
\begin{split}
\D(z_1,z_2,\ldots , z_n)_{|V}&=
\frac{1}{\mathbbm{1} -\phi_1 z_1 - \phi_2 z_2 - \cdots - \phi_n z_n}\\[5pt]
& = \mathbbm{1} + \phi_1 z_1 + \cdots + \phi_n z_n + ( \phi_1 z_1 + \cdots + \phi_n z_n)^2 + \cdots\\[5pt]
&=\mathbbm{1}+\sum_{i\geq 1} ( \phi_1 z_1 + \cdots + \phi_n z_n)^i.
\end{split}
\end{equation}
Let  $u\in V$ be an arbitrary non null vector. Integration by parts  \eqref{for2.3} gives 
\begin{equation}
\OvD(z_1,\ldots, z_n)(\D(z_1,\ldots,z_n)u\w v)=u\w \OvD(z_1,\ldots,z_n)v.\label{eq3:fnl}
\end{equation}
Corollary \ref{cor1} implies that the coefficients of $z_1z_2\cdots z_n$ of the second member must be zero, i.e., explicitly
\begin{equation}
\OvD(z_1,\ldots, z_n)(\D(z_1,\ldots,z_n)u\w v)=0.\label{eq4:15}
\end{equation}	
Now  the coefficients of the monomial $z_1z_2\cdots z_n$ occurring  on the left hand side of \eqref{eq3:fnl}  are all of the form
\begin{equation}
(-1)^k\ovD_{\bfe_{i_1}+\cdots+\bfe_{i_k}}(D_{\bfe_{i_{k+1}}+\cdots+\bfe_{i_n}}u\w v),\label{eq:red}
\end{equation}
where $(i_1,\ldots,i_k,i_{k+1},\ldots,i_n)$ is any b\textit{•}jection of the set $\{1,2,\ldots,n\}$ such that $i_{k+1}<\cdots<i_n$. 
Fixing $k\geq 1$, the sum of all terms like \eqref{eq:red} occurring in \eqref{eq4:15} could be uniformly written as
\begin{equation}
\sum_{k=1}^n(-1)^k\ovD_{\bfe_{\sigma(1)}+\cdots+
\bfe_{\sigma(k)}}(D_{ \bfe_{i_{\sigma(k+1)}}+\cdots+\bfe_{i_{\sigma(k)}}}u
\w v),\label{eq:above}
\end{equation}
with the effect that any expression \eqref{eq:red} occurring only once in \eqref{eq4:15}, occurs with multiplicity $k!$ in \eqref{eq:above}. So we have to divide by $k!$ factorial to avoid redundance. It follows that \eqref{eq4:15} implies

\begin{equation}\label{eq:eq311}
\sum_{\sigma}\sum_{k=1}^n{(-1)^k\over k!}\ovD_{\bfe_{\sigma(1)}+\cdots+
\bfe_{\sigma(k)}}(D_{ \bfe_{i_{\sigma(k+1)}}+\cdots+\bfe_{i_{\sigma(k)}}}u
\w v)=0,
\end{equation}
where the first sum is over all $\sigma \in S_n$  such that $\sigma(k+1)<
\cdots<\sigma(n)$, whose cardinality is precisely $k!$.
	
Since  $D_{\bfe_{\sigma(k+1)}+\cdots+\bfe_{\sigma(n)}}u\w v\in \bw^nV$, it is an eigenvector of 
$\ovD_{\bfe_{\sigma(1)}+\cdots+\bfe_{\sigma(k)}}$
with respect to the eigenvalue $\tau_{\bfe_{\sigma(1)}+\cdots+\bfe_{\sigma(k)}}(\uphi)\in\KK$.
On the other hand
$D_{\bfe_{\sigma(k+1)}+\cdots+\bfe_{\sigma(n)}}$
is the coefficient of $z_{\sigma(k+1)}\cdots z_{\sigma(n)}$ in $(\phi_1z_1+
\cdots+\phi_nz_n)^{n-k}$, which is precisely (keeping into account that the 
composition of endomorphisms is not commutative)
\[\sum_{\gamma\in S_{n-k}}\phi_{\gamma(\sigma(k+1))}\cdot\ldots\cdot \phi_{\gamma(\sigma(n))},\]
where we are thinking $S_{n-k}$ precisely as the group of bijections of the set\linebreak  
$\{\sigma(k+1),\ldots,\sigma(n)\}$.
Therefore, formula \eqref{eq:eq311} can be rewritten as
\begin{equation}\label{eq:eq312}
\begin{split}
0&=\sum_{k=0}^n{(-1)^k\over k!}\sum_{\sigma\in S_n}\tau_{\bfe_{\sigma(1)}+\cdots+\bfe_{\sigma(k)}}(\uphi)
\cdot (\phi_{\sigma(k+1)}\circ \cdots  \circ\phi_{\sigma(n)}u\w v)\\
&=\left[\sum_{k=0}^n(-1)^k\dfrac{1}{k!}
\sum_{\sigma\in S_n}\tau_{\bfe_{\sigma(1)}+\cdots+\bfe_{\sigma(k)}}(\uphi)
\cdot \phi_{\sigma(k+1)}\circ \cdots \circ \phi_{\sigma(n)}u\right]\w v,
\end{split}
\end{equation}
using the fact that $\tau_{\bfe_{\sigma(1)},\cdots,\bfe_{\sigma(k)}}(\uphi)\in \KK$ and the multilinearity of the wedge product.
Since \eqref{eq:eq312} holds for each choice of $v\in \bw^{n-1}V$, formula \eqref{eq:formnthm} follows.
\end{proof}
\begin{cor}
For all ordered $n$-tuples of  $\KK$-valued $n\times n$ matrices $\a=(A_1,\ldots, A_n)$ the following identity holds:
\[	\sum_{k=0}^n(-1)^k\dfrac{1}{k!}\sum_{\sigma\in S_n}\tau_{\bfe_{\sigma(1)}+\cdots+\bfe_{\sigma(k)}}(\a)\cdot (A_{\sigma(k+1)} \cdots A_{\sigma(n)})=0,
\]
where $\tau_{\m}(\a)$ denotes the $\m$-trace of the $n$-tuples of matrices thought of as endomorphisms of $\KK^n$.
\end{cor}
\section{Applications and Examples}\label{sec:sec5}
\begin{exam}\label{extraces}
Let us consider the  ordered triple $\a:=(A,B,C)$, where    $A=(a_{ij})$, $B=(b_{ij})$, $C=(c_{ij})$ are  $3\times 3$ square matrices  thought of as endomorphisms of 
$\KK^3$.
Let  $\OvD(\uz):=\OvD_{\a}(z_1,z_2,z_3):\bw\KK^3\sra(\bw\KK^3)[z_1,z_2,z_3]$ be the unique multivariate HS-derivation on $\bw\KK^3$ such that
\[\OvD(\uz)_{|V}= \mathbbm{1}-Az_1-Bz_2-Cz_3.\]
If  $1:=\bfe_1\w\bfe_2\w\bfe_3$, is the canonical basis of $\bw^3\KK^3$,
one has
\begin{equation}
\setlength\arraycolsep{1pt}
\begin{array}{cll}
&&\OvD(\uz)(\bfe_1\w\bfe_2\w\bfe_3)=\displaystyle{\sum}_{\m\in\N^3}\ovD_{\m}(\bfe_1\w\bfe_2\w\bfe_3) z^\m=\OvD(\uz)\bfe_1\w\OvD(\uz)\bfe_2\w \OvD(\uz)\bfe_3\cr\cr
&=&(\bfe_1-A\bfe_1z_1-B\bfe_1z_2-C\bfe_1z_3)\w(\bfe_2-A\bfe_2z_1-B\bfe_2z_2-C\bfe_2z_3)\cr\cr
&&\w(\bfe_3-A\bfe_3\cdot z_1-B\bfe_3\cdot z_2-C\bfe_3\cdot z_3).\label{eq5:17}
\end{array}
\end{equation}
Then $\ovD_{\m}(\bfe_1\w\bfe_2\w\bfe_3)$ is the coefficient of $\uz^\m$ in the right hand side of the expression~\eqref{eq5:17}. For example
$\ovD_{(2,1,0)}(\bfe_1\w\bfe_2\w\bfe_3)$
is the coefficient of $z_1^2z_2$, i.e., precisely:
\[
\ovD_{(2,1,0)}(\bfe_1\w\bfe_2\w\bfe_3)=
A\bfe_1\w B\bfe_2\w A\bfe_3+A\bfe_1\w A\bfe_2\w B\bfe_3+B\bfe_1\w A\bfe_2\w A\bfe_3.
\]
In particular, using the definition of determinant of a matrix, one has
\[
\tau_{(2,1,0)}(\Acal)=\det(B\bfe_1, A\bfe_2, A\bfe_3)+\det(A\bfe_1, A\bfe_2,B\bfe_3)
+\det(A\bfe_1, B\bfe_2, A\bfe_3).
\]

In the same fashion, one computes all the remaining traces $\tau_{(i_1.i_2,i_3)}$, which are zero if $i_1+i_2+i_3\geq 4$. The complete list follows, where $\tr(A)$ of a matrix $A$ is the usual ordinary trace computed as the sum of the elements of the first diagonal.
	
\[\tau_{(0,0,0)}(\Acal)=1,\]
\[\tau_{(1,0,0)}(\a)=\tr(A_1),\qquad \tau_{(0,1,0)}(\a)=\tr(A_2), \qquad \tau_{(0,0,1)}(\a)=\tr(A_3)\]
	
\centerline{and}
\bigskip
\begin{center}
\begin{tabular}{|c|cll|}\hline&&&\\
&&
$\begin{vmatrix}
a_{11}&b_{12}&0\\
a_{21}&b_{22}&0\\
a_{31}&b_{32}&1
\end{vmatrix}+
\begin{vmatrix}
a_{11}&0&b_{13}\\
a_{21}&1&b_{23}\\
a_{31}&0&b_{33}
\end{vmatrix}+
\begin{vmatrix}
1&a_{12}&b_{13}\\
0&a_{22}&b_{23}\\
0&a_{32}&b_{33}
\end{vmatrix}$&\\
$\tau_{(1,1,0)}(\a)=$&&&\\
&+&$\begin{vmatrix}
b_{11}&a_{12}&0\\
b_{21}&a_{22}&0\\
b_{31}&a_{32}&1
\end{vmatrix}+
\begin{vmatrix}
b_{11}&0&a_{13}\\
b_{21}&1&a_{23}\\
b_{31}&0&a_{33}
\end{vmatrix}+
\begin{vmatrix}
1&b_{12}&a_{13}\\
0&b_{22}&a_{23}\\
0&b_{32}&a_{33}
\end{vmatrix}$&\\ 
&&&
\\
\hline
\end{tabular}
\end{center}

\begin{center}
\begin{tabular}{|c|cll|}\hline&&&\\
&&
$\begin{vmatrix}
a_{11}&c_{12}&0\\
a_{21}&c_{22}&0\\
a_{31}&c_{32}&1
\end{vmatrix}+
\begin{vmatrix}
a_{11}&0&c_{13}\\
a_{21}&1&c_{23}\\
a_{31}&0&c_{33}
\end{vmatrix}+
\begin{vmatrix}
1&a_{12}&c_{13}\\
0&a_{22}&c_{23}\\
0&a_{32}&c_{33}
\end{vmatrix}$&
\\ $\tau_{(1,0,1)}(\a)=$&&&\\
&+&
$\begin{vmatrix}
c_{11}&a_{12}&0\\
c_{21}&a_{22}&0\\
c_{31}&a_{32}&1
\end{vmatrix}+
\begin{vmatrix}
c_{11}&0&a_{13}\\
c_{21}&1&a_{23}\\
c_{31}&0&a_{33}
\end{vmatrix}+
\begin{vmatrix}
1&c_{12}&a_{13}\\
0&c_{22}&a_{23}\\
0&c_{32}&a_{33}
\end{vmatrix}$&\\ &&&\\
\hline&&&\\
&&$\begin{vmatrix}
b_{11}&c_{12}&0\\
b_{21}&c_{22}&0\\
b_{31}&c_{32}&1
\end{vmatrix}+
\begin{vmatrix}
b_{11}&0&c_{13}\\
b_{21}&1&c_{23}\\
b_{31}&0&c_{33}
\end{vmatrix}+
\begin{vmatrix}
1&b_{12}&c_{13}\\
0&b_{22}&c_{23}\\
0&b_{32}&c_{33}
\end{vmatrix}$&\\
$\tau_{(0,1,1)}(\a)=$&&&\\
&+&
$\begin{vmatrix}
c_{11}&b_{12}&0\\
c_{21}&b_{22}&0\\
c_{31}&b_{32}&1
\end{vmatrix}+
\begin{vmatrix}
c_{11}&0&b_{13}\\
c_{21}&1&b_{23}\\
c_{31}&0&b_{33}
\end{vmatrix}+
\begin{vmatrix}
1&c_{12}&b_{13}\\
0&c_{22}&b_{23}\\
0&c_{32}&b_{33}
\end{vmatrix}$&\\
&&&\\\hline&&&\\
$\tau_{(2,0,0)}(\a)=$&&$\begin{vmatrix}
a_{11}&a_{12}&0\\
a_{21}&a_{22}&0\\
a_{31}&a_{32}&1
\end{vmatrix}+
\begin{vmatrix}
a_{11}&0&a_{13}\\
a_{21}&1&a_{23}\\
a_{31}&0&a_{33}
\end{vmatrix}+
\begin{vmatrix}
1&a_{12}&a_{13}\\
0&a_{22}&a_{23}\\
0&a_{32}&a_{33}
\end{vmatrix}$&\\ &&&\\
\hline &&&\\
$\tau_{(0,2,0)}(\a)=$&&$\begin{vmatrix}
b_{11}&b_{12}&0\\
b_{21}&b_{22}&0\\
b_{31}&b_{32}&1
\end{vmatrix}+
\begin{vmatrix}
b_{11}&0&b_{13}\\
b_{21}&1&b_{23}\\
b_{31}&0&b_{33}
\end{vmatrix}+
\begin{vmatrix}
1&b_{12}&b_{13}\\
0&b_{22}&b_{23}\\
0&b_{32}&b_{33}
\end{vmatrix}$&\\
&&&\\ 
\hline&&& \\ 
$\tau_{(0,0,2)}(\a)=$&&$\begin{vmatrix}
c_{11}&c_{12}&0\\
c_{21}&c_{22}&0\\
c_{31}&c_{32}&1
\end{vmatrix}+
\begin{vmatrix}
c_{11}&0&c_{13}\\
c_{21}&1&c_{23}\\
c_{31}&0&c_{33}
\end{vmatrix}+
\begin{vmatrix}
1&c_{12}&c_{13}\\
0&c_{22}&c_{23}\\
0&c_{32}&c_{33}
\end{vmatrix},$&\\&&& \\ \hline&&&\\
$\tau_{(2,1,0)}(\a)=$&&$\begin{vmatrix}
a_{11}&a_{12}&b_{13}\\
a_{21}&a_{22}&b_{23}\\
a_{31}&a_{32}&b_{33}
\end{vmatrix}+
\begin{vmatrix}
a_{11}&b_{12}&a_{13}\\
a_{21}&b_{22}&a_{23}\\
a_{31}&b_{32}&a_{33}
\end{vmatrix}+
\begin{vmatrix}
b_{11}&a_{12}&a_{13}\\
b_{21}&a_{22}&a_{23}\\
b_{31}&a_{32}&a_{33}
\end{vmatrix}$&\\&&& 
\\ \hline&&&\\
$\tau_{(1,2,0)}(\a)=$&&$\begin{vmatrix}
b_{11}&b_{12}&a_{13}\\
b_{21}&b_{22}&a_{23}\\
b_{31}&b_{32}&a_{33}
\end{vmatrix}+
\begin{vmatrix}
b_{11}&a_{12}&b_{13}\\
b_{21}&a_{22}&b_{23}\\
b_{31}&a_{32}&b_{33}
\end{vmatrix}+
\begin{vmatrix}
a_{11}&b_{12}&b_{13}\\
a_{21}&b_{22}&b_{23}\\
a_{31}&b_{32}&b_{33}
\end{vmatrix}$& \\&&&\\
\hline
\end{tabular}
\end{center}
\begin{center}
\begin{tabular}{|c|cll|}\hline&&&\\
$\tau_{(2,0,1)}(\a)=$&$\begin{vmatrix}
a_{11}&a_{12}&c_{13}\\
a_{21}&a_{22}&c_{23}\\
a_{31}&a_{32}&c_{33}
\end{vmatrix}+
\begin{vmatrix}
a_{11}&c_{12}&a_{13}\\
a_{21}&c_{22}&a_{23}\\
a_{31}&c_{32}&a_{33}
\end{vmatrix}+
\begin{vmatrix}
c_{11}&a_{12}&a_{13}\\
c_{21}&a_{22}&a_{23}\\
c_{31}&a_{32}&a_{33}
\end{vmatrix}$&&\\&&&\\
\hline &&&\\
$\tau_{(1,0,2)}(\a)=$&$\begin{vmatrix}
c_{11}&c_{12}&a_{13}\\
c_{21}&c_{22}&a_{23}\\
c_{31}&c_{32}&a_{33}
\end{vmatrix}+
\begin{vmatrix}
c_{11}&a_{12}&c_{13}\\
c_{21}&a_{22}&c_{23}\\
c_{31}&a_{32}&c_{33}
\end{vmatrix}+
\begin{vmatrix}
a_{11}&c_{12}&c_{13}\\
a_{21}&c_{22}&c_{23}\\
a_{31}&c_{32}&c_{33}
\end{vmatrix}$&&\\  &&&\\
\hline &&&\\
$\tau_{(0,2,1)}(\a)=$&$\begin{vmatrix}
b_{11}&b_{12}&c_{13}\\
b_{21}&b_{22}&c_{23}\\
b_{31}&b_{32}&c_{33}
\end{vmatrix}+
\begin{vmatrix}
b_{11}&c_{12}&b_{13}\\
b_{21}&c_{22}&b_{23}\\
b_{31}&c_{32}&b_{33}
\end{vmatrix}+
\begin{vmatrix}
c_{11}&b_{12}&b_{13}\\
c_{21}&b_{22}&b_{23}\\
c_{31}&b_{32}&b_{33}
\end{vmatrix}$&&\\&&&\\ \hline&&& \\
$\tau_{(0,1,2)}(\a)=$&$\begin{vmatrix}
c_{11}&c_{12}&b_{13}\\
c_{21}&c_{22}&b_{23}\\
c_{31}&c_{32}&b_{33}
\end{vmatrix}+
\begin{vmatrix}
c_{11}&b_{12}&c_{13}\\
c_{21}&b_{22}&c_{23}\\
c_{31}&b_{32}&c_{33}
\end{vmatrix}+
\begin{vmatrix}
b_{11}&c_{12}&c_{13}\\
b_{21}&c_{22}&c_{23}\\
b_{31}&c_{32}&c_{33}
\end{vmatrix}$&&\\ &&&
\\ \hline&&&
\\
&$\begin{vmatrix}
a_{11}&b_{12}&c_{13}\\
a_{21}&b_{22}&c_{23}\\
a_{31}&b_{32}&c_{33}
\end{vmatrix}+
\begin{vmatrix}
a_{11}&c_{12}&b_{13}\\
a_{21}&c_{22}&b_{23}\\
a_{31}&c_{32}&b_{33}
\end{vmatrix}+
\begin{vmatrix}
b_{11}&a_{12}&c_{13}\\
b_{21}&a_{22}&c_{23}\\
b_{31}&a_{32}&c_{33}
\end{vmatrix}+$&&\\
$\tau_{(1,1,1)}(\a)=$&&&\\
&$\begin{vmatrix}
b_{11}&c_{12}&a_{13}\\
b_{21}&c_{22}&a_{23}\\
b_{31}&c_{32}&a_{33}
\end{vmatrix}+
\begin{vmatrix}
c_{11}&b_{12}&a_{13}\\
c_{21}&b_{22}&a_{23}\\
c_{31}&b_{32}&a_{33}
\end{vmatrix}+
\begin{vmatrix}
c_{11}&a_{12}&b_{13}\\
c_{21}&a_{22}&b_{23}\\
c_{31}&a_{32}&b_{33}
\end{vmatrix}$&&\\ &&&\\ \hline
\end{tabular}
\end{center}
	
	and, finally
	\[
	\tau_{(3,0,0)}(\a)=\det(A),\qquad
	\tau_{(0,3,0)}(\a)=\det(B),\quad
	\tau_{(0,0,3)}(\a)=\det(C). 
	\]
\end{exam}

\begin{nota}\label{finnot}
Let $\a=(A_1, \cdots,A_n)$ be $n\times n$ matrices.
For each $ 1\leq k\leq n $, let
$\d_{\bfe_{i_1}+\cdots+\bfe_{i_k}}(\a)$ be the scalar
defined as follows
\[
\d_{\bfe_{i_1}+\cdots+\bfe_{i_k}}(\a)=
\det(\bfe_1, \cdots, A_{i_1}\bfe_{i_1}, \cdots , 
A_{i_k}\bfe_{i_k}, \cdots, \bfe_n).
\]
\end{nota}
\noindent
For example, in case of a triple $\a:=(A,B,C)$ of $3\times 3$ matrices:
\begin{equation*}
\begin{array}{ll}
\d_{(1,0,0)}(\a)=\det(A\bfe_1,\bfe_2,\bfe_3),\quad &\d_{(0,1,0)}(\a)=
\det(\bfe_1, B\bfe_2,\bfe_3),\cr\cr
\d_{(0,0,1)}(\a)=\det(\bfe_1,\bfe_2, C\bfe_3),
\quad &\d_{(1,1,0)}(\a)=\det(A\bfe_1, B\bfe_2,\bfe_3),\cr\cr
\d_{(1,0,1)}(\a)=\det(A\bfe_1,\bfe_2, C\bfe_3),\quad
&\d_{(0,1,1)}(\a)=\det(\bfe_1, B\bfe_2, C\bfe_3),
\end{array}
\end{equation*}
and, finally, 
$
\d_{(1,1,1)}(\a)=\det(A\bfe_1, B\bfe_2, C\bfe_3),
$
where $(\bfe_1,\bfe_2,\bfe_3)$ denotes, as usual, the canonical basis, so that, e.g., $A\bfe_j$ denotes the $j$-th column of $A$.

\begin{exam}\label{ex:exam52} Notation as in \ref{finnot}. Consider the tri-linear form
$$
\left\{\begin{matrix}\KK^{3\times 3}\times \KK^{3\times 3}\times\KK^{3\times 3}&\longrightarrow &\KK^{3\times 3}\cr\cr
(A,B,C)&\longmapsto& A\star B\star C\end{matrix}\right.
$$
where
\begin{eqnarray}
A\star B\star C &=&ABC\cr\cr
&-&\d_{(1,0,0)}(\a)BC-\d_{(0,1,0)}(\a)CA-\d_{(0,0,1)}(\a)AB\cr \cr
&+&\d_{(1,1,0)}(\a)C+\d_{(0,1,1)}(\a)A+\d_{(1,0,1)}(\a)B\cr \cr
&-&\det(A\bfe_1,B\bfe_2,C\bfe_3)\cdot \mathbbm{1}_{3\times 3}.\label{eq:trilin}
\end{eqnarray}
By applying  the very definition of the traces, taking into account the occurring permutations, the coefficients of the matrix product $BC$ and $CB$ in
 the sum
\begin{equation}\label{eq5:lsteq}
A\star B\star C+
B\star A\star C+
C\star A\star B+
B\star C\star A+
A\star C\star B+
C\star B\star A,
\end{equation}
are easily seen to be, respectively:
$$
\left(\d_{(1,0,0)}(A,B,C)+\d_{(0,1,0)}(C,A,B)+\d_{(0,0,1)}(B,C,A)\right)BC=\tr(A)BC
$$
and
$$
\left(\d_{(1,0,0)}(A,C,B)+\d_{(0,1,0)}(B,A,C)+\d_{(0,0,1)}(C,B,A)\right)CB=\tr(A)CB,
$$
having noticed that
\[\d_{(1,0,0)}(A,B,C)+\d_{(0,1,0)}(C,A,B)+\d_{0,0,1}(B,C,A)=\tr(A)={\tau_{(1)}(A)}.
\]
Therefore, by interchanging the roles of $A,B,C$, one sees that the quadratic part in the sum \eqref{eq5:lsteq} is given by:
\begin{equation}\label{eq:trs}
\tr(A)(BC+CB)+\tr(B)(AC+CA)+\tr(C)(AB+BA).
\end{equation}
We know that $\tr(A)=\tau_{(1,0,0)}(\a)$, $\tr(B)=\tau_{(0,1,0)}(\a)$, and $\tr(C)=\tau_{(0,0,1)}(\a)$. Therefore, \eqref{eq:trs} is equal to
\[
\tau_{(1,0,0)}(\a)(BC+CB)+\tau_{(0,1,0)}(\a)(AC+CA)+\tau_{(0,0,1)}(\a)(AB+BA).
\]

\noindent
A similar analysis shows that the coefficients of $A,B,C$ occurring in expression \eqref{eq5:lsteq} are given respectively by $\tau_{(0,1,1,)}(\Acal)$, $\tau_{(1,0,1,)}(\Acal)$ and $\tau_{(1,0,0,)}(\Acal)$. Just to show how the computations go on, the coefficient of $C$ would be:
\begin{equation*}
\begin{split}
\big(&\det((A\bfe_1,B\bfe_2,\bfe_3)+\det(A\bfe_1,\bfe_2,B\bfe_3)
+\det(B\bfe_1,A\bfe_2,\bfe_3)+\\
&\det(\bfe_1,A\bfe_2,B\bfe_3)
+\det(\bfe_1,B\bfe_2,A\bfe_3)+\det(B\bfe_1,\bfe_2,A\bfe_3)\big)C,
\end{split}
\end{equation*}
which is exactly the expression of $\tau_{(1,1,0)}(\a)C$. Finally, 
the coefficient of the identity matrix $\mathbbm{1}_{3\times 3}$ occurring in expression~\eqref{eq5:lsteq} is:
\begin{equation*}
\begin{split}
&\det(A\bfe_1,B\bfe_2,C\bfe_3)+\det(A\bfe_1,C\bfe_2,B\bfe_3)+
\det(B\bfe_1,A\bfe_2,C\bfe_3)+\det(B\bfe_1,C\bfe_2,A\bfe_3)\\[5pt]
&+\det(C\bfe_1,B\bfe_2,A\bfe_3)+\det(C\bfe_1,A\bfe_2,B\bfe_3),
\end{split}
\end{equation*}
which is equal to $\tau_{(1,1,1)}(\a)$.
The sum  \eqref{eq5:lsteq} can be then rewritten as:
\begin{equation*}
\setlength\arraycolsep{1pt}
\begin{array}{ccl}
ABC&+&BAC+CAB+BCA+ACB+CBA+\tau_{(1,0,0)}(\a)(BC+CB)\cr\cr 
&+&\tau_{(0,1,0)}(\a)(AC+CA)
+\tau_{(0,0,1)}(\a)(AB+BA)+\tau_{(1,1,0)}(\a)C\cr\cr 
&+&\tau_{(1,0,1)}(\a)B+\tau_{(0,1,1)}(\a)A
+\tau_{(1,1,1)}(\a)\mathbbm{1}_{3\times 3},
\end{array}
\end{equation*}
which, by Theorem \ref{thm:thm47}, is zero. 
In other words,  Theorem \ref{thm:thm47} can be phrased as follows.
{\em The tri-linear form given by \eqref{eq:trilin} is totally anti-symmetric.}
\end{exam}
\begin{exam}
	Let $\underline{\Bcal}=(A,B)$ a pair of $3\times 3$ matrices $A(i,j)=a_{ij}$ and $B(i,j)=b_{ij}$, thought of  as endomorphisms on $V=\KK^3$.
	Consider the unique HS--derivation 
\[\OvD(z_1,z_2)\colon \bigwedge V
\rightarrow \bigwedge V\llb z_1,z_2 \rrb\]
	such that
$\OvD(z_1,z_2)=\mathbbm{1} -Az_1-Bz_2,
	$  and let us look at the coefficients of the third degree of \eqref{prop1} applied to an element of $\bw^3V$. 
	In particular we look at the coefficients of $ z_1^2z_2 $ in \eqref{for2.3}, and compute the trace $\tau_{(2,0,0)}(\underline{\Bcal})$. 
	\begin{equation}
	\begin{split}
	\OvD(\uz)(\bfe_1\w &\bfe_2\w \bfe_3)=\OvD \bfe_1 \w\OvD \bfe_2\w\OvD \bfe_3=\bfe_1\w \bfe_2\w \bfe_3\\
	&-(A\bfe_1\w \bfe_2\w \bfe_3+\bfe_1\w A\bfe_2\w \bfe_3+\bfe_1\w \bfe_2\w A\bfe_3)z_1+\cdots\\
	&-(A\bfe_1\w A\bfe_2\w B\bfe_3+A\bfe_1\w B\bfe_2\w A\bfe_3+B\bfe_1\w A\bfe_2\w A\bfe_3)z_1^2 z_2\\
	&+\cdots
	\end{split}
	\end{equation}
	from which:
	\begin{equation}
	\begin{split}
	&\tau_{(2,1,0)}(\underline{\Bcal})(\bfe_1\w \bfe_2\w \bfe_3)=\\
	& (A\bfe_1\w A\bfe_2\w B\bfe_3+A\bfe_1\w B\bfe_2\w A\bfe_3+B\bfe_1\w A\bfe_2\w A\bfe_3).
	\end{split}
	\end{equation}
	The coefficient $\tau_{(2,0,0)}(\Bcal)$ also is computed similarly.
By an argument similar to Theorem \ref{thm:thm47} one obtains:
	\begin{equation}\label{eq17}
	\begin{split}
	&(A^2B+AB A+B A^2) - \tau_{(1,0,0)}(\underline{\Bcal})(AB+B A)\\[5pt]
	& - \tau_{(0,1,0)}(\underline{\Bcal}) A^2 
	+\tau_{(2,0,0)}(\underline{\Bcal}) B+
	\tau_{(1,1,0)}(\underline{\Bcal}) A
	-\tau_{(2,1,0)}(\underline{\Bcal}) \mathbbm{1}_{3\times 3}=0.
	\end{split}
	\end{equation}
\end{exam}
\begin{exam}
In particular, if in the example \ref{extraces} we put $A=C$, then:
\begin{equation}
A\star A\star B+B\star A\star A+A\star B\star A=0,\label{eq:fin}
\end{equation}
which is exactly the content of \eqref{eq17}. Once again, putting $A=B$ in 
\eqref{eq:fin} (i.e. $A=B=C$ in \eqref{eq5:lsteq}) 
one obtains the Cayley--Hamilton theorem for $3\times 3$ matrices (by clearing a factor $3$):
\begin{equation}\label{eq:final}
A\star A\star A=A^3-\text{tr}(A)A^2+\tau_{(2,0,0)}(\a)A - \det(A)\mathbbm{1}_{3\times 3}=0.
\end{equation}
\end{exam}

\medskip
\noindent
{\bf Acknowledgement.}
Warm words of gratitude are due to the Referee. Her/his patient and carefully reading of the first versions of the manuscript helped me to greatly improve the exposition.

\end{document}